\documentclass[a4paper,10pt]{article}
\usepackage[utf8]{inputenc}
\usepackage{amsmath, amsfonts, amsthm,amssymb,  bbm, dsfont, lmodern}
\usepackage{math}
\usepackage{hyperref, geometry,mathtools} 
\mathtoolsset{showonlyrefs}

\usepackage[usenames, dvipsnames]{color}
\newtheorem{lemma}{Lemma}[section]
\newtheorem{theorem}[lemma]{Theorem}
\newtheorem{proposition}[lemma]{Proposition}

\newtheorem{definition}[lemma]{Definition}
\theoremstyle{remark}
\newtheorem{remark}[lemma]{Remark}

 \usepackage{mathrsfs}

\def\beq{\begin{equation}}   \def\eeq{\end{equation}}
\def\bea{\begin{eqnarray}}  \def\eea{\end{eqnarray}}

\newcommand{\pa}{\partial}
\newcommand{\hos}{{\,} }

\newcommand{\cl}{{\rm cl}}

\renewcommand{\bar}{\overline}

\renewcommand{\fA}{\mathsf{A}}
\renewcommand{\fB}{\mathsf{B}}

\renewcommand{\fH}{\mathsf{H}}
\renewcommand{\fK}{\mathsf{K}}

\renewcommand{\fV}{\mathsf{V}}

\newcommand{\sV}{\mathscr{V}}
\newcommand{\sA}{\mathscr{A}}
\renewcommand{\sA}{\mathscr{A}}

\newcommand{\Opw}[1]{{\rm Op}^W\!\left(#1\right)}
\newcommand{\Opaw}[1]{{\rm Op}^{aw}\!\left(#1\right)}

\title{Generic transporters  for the  linear time dependent  quantum  Harmonic oscillator on $\R$}

\author{
A. Maspero\footnote{ International School for Advanced Studies (SISSA), Via Bonomea 265, 34136, Trieste, Italy \newline
 \textit{Email: } \texttt{alberto.maspero@sissa.it}}
}

\numberwithin{equation}{section}

\begin{document}
\maketitle

\begin{abstract}
In this paper we consider the linear, time dependent quantum Harmonic  Schr\"odinger equation
$\im \pa_t u= \frac{1}{2} ( - \pa_x^2 + x^2) u +  V(t, x, D)u$, $x \in \R$, where  $V(t,x,D)$ is classical pseudodifferential operator of order 0, selfadjoint, and  $2\pi$ periodic in time.\\
We give sufficient conditions on the principal symbol of
$V(t,x,D)$  ensuring the existence of  solutions displaying infinite time growth of Sobolev norms.
These conditions are generic in the Fr\'echet space of symbols. 
This  shows  that generic, classical pseudodifferential, $2\pi$-periodic perturbations provoke unstable dynamics.
The proof builds on the results of \cite{Mas21} and it is based on pseudodifferential normal form and  local energy decay estimates. 
These last are proved exploiting Mourre's positive commutator theory.
\end{abstract}

\section{Introduction and main result}
\label{sec:app}
In this paper we study the perturbed   quantum harmonic oscillator on $\R$
\begin{align}
\label{har.osc}
& \im \pa_t u= \frac{1}{2} ( - \pa_x^2 + x^2) u +  V(t, x, D)u  , \quad x \in \R \ . 
\end{align}
 We shall always assume that  $V(t,x,D)$ is a bounded operator, selfadjoint and  $2\pi$-periodic in time.
Our goal is to construct solutions  exhibiting unstable behavior in the form of forward energy cascade.
Precisely,  we shall exhibit solutions of \eqref{har.osc} whose  $\cH^r$-Sobolev norms, $r>0$, grows unbounded in time:
\begin{equation}
\label{limsup}
\limsup_{t \to \infty} \norm{u(t)}_{r}  = +\infty  \ ; 
\end{equation}
 here we denoted, for any $r \in \R$,
\begin{equation}
\label{Hr}
\cH^r:= \{ u \in L^2(\R,\C) \colon \ \  \norm{u}_r:= \|H_0^r \, u \|_{L^2}< \infty \} , \qquad H_0 := \frac{1}{2} ( - \pa_x^2 + x^2) \ .
\end{equation}
Note that, when $V = 0$, the unperturbed evolution $e^{- \im t H_0}$ preserves all norms $\|\cdot \|_r$ for all times and no energy cascade occurs. 
So the  question is whether one can construct an operator $V$ producing  unbounded orbits.

To formalize  this concept we shall say  (following \cite{Mas21}) that  $V(t,x,D)$ is   a {\em transporter}
if \eqref{har.osc} has at least one solution fulfilling \eqref{limsup} for some $r >0$.

In the last few years several  transporters for \eqref{har.osc} were constructed by Delort \cite{del}, Bambusi-Gr\'ebert-M.-Robert \cite{BGMR1}, M. \cite{Mas19}, Faou-Raphael \cite{FaouRaph}, Liang, Zhao and Zhou \cite{LZZ}, M. \cite{Mas21}, Thomann \cite{Thomann},  Luo, Liang and Zhao \cite{LLZ}; we will comment more about these results later on.

For the moment, let us note that all these above are  {\em examples} of transporters, and it is not clear neither how to determine if a given operator $V(t,x,D)$ is a transporter nor what happens for generic operators. 
More precisely, the following  questions are  open and, we believe, of great interest (not only for the linear theory, but also for applications to nonlinear systems):
\begin{itemize}
\item[(Q1)] Given an operator $V(t,x,D)$, can we identify sufficient conditions guaranteeing it to be a transporter?
\item[(Q2)] Are transporters rare or common? In other words, how generic are transporters?
\end{itemize}
In this paper we answer both questions, at least in case   $V(t,x,D)$  belongs to the class of  classical pseudodifferential operators of order 0. 
We identify, for the first time,  explicit, sufficient conditions on the principal symbol of $V(t,x,D)$ which guarantee  the operator to be a transporter. 
We  show that these conditions are fulfilled for {\em generic} symbols, meaning for a set which is open and dense in the Fr\'echet topology of  symbols.\\
As a conclusion, we obtain that  generic, $2\pi$-time periodic, classical pseudodifferential perturbations of order 0 produce unbounded orbits -- a fact which is, in our opinion, somewhat surprising.

The conditions we identify on $V(t,x,D)$ are actually very simple. 
Let $v_0(t,x,\xi)$ be its  principal symbol.
Assume it to be  a positively homogeneous function of degree 0 (see Definition \ref{symbol.ao2}), and denote by
$\la v_0 \ra(x,\xi)$  its {\em resonant average} with respect to the classical flow 
\begin{equation}
\label{flow.ham}
\phi^t(x, \xi) := (x \cos t + \xi \sin t, -x \sin t + \xi \cos t)
\end{equation}
 of the harmonic oscillator $h_0(x,\xi):= \frac12 (x^2 + \xi^2)$, i.e. 
\begin{equation}
\label{res.av0}
\la v_0 \ra(x,\xi) :=  \frac{1}{2\pi}\int_0^{2\pi} v_0(t, \phi^t(x,\xi)) \,  \di t  \ . 
\end{equation}
Our main Theorem \ref{thm:main} shows that, if the Poisson bracket between $\la v_0 \ra$ and $h_0$  does not vanish identically outside the origin, i.e.
\begin{equation}
\label{cond}
\{ \la v_0 \ra, h_0 \} \not \equiv 0 \ \  \ \ \mbox{ in }  x^2  + \xi^2 \geq 1  \ , 
\end{equation}
then  $V(t,x,D)$ is a transporter.
Here we use the convention that 
\begin{equation}
\label{}
\{ f, g \} := \pa_\xi f \cdot \pa_x g - \pa_x f \cdot \pa_\xi g \ . 
\end{equation}
The proof of this result, that we will describe at the end of the section,  builds on the theory developed in \cite{Mas21}, and it is based on a combination of pseudodifferential normal form and a dispersive mechanism in the energy space.
The dispersion is quantitatively described by local energy decay estimates, which in turn are proved exploiting Mourre's theory of positive commutators. 

We remark that in \cite{Mas21} we were already able to 
apply some abstract results to  equation \eqref{har.osc}; however
  we were able only to deal with operators belonging to   the special class of smooth T\"oplitz operators\footnote{pseudodifferential operators whose matrix elements (computed on the basis of the Hermite functions) are  constant on the diagonals and decaying fast enough off diagonal}.
On the contrary, the main improvement of the current paper is to deal with  
generic  classical pseudodifferential operators of order 0. 

Let us now state precisely our results.

\subsection{Main result}

We first define the class of symbols we use.

\begin{definition}\label{symbol.ao2}
$(i)$ A function $f$ is a {\em  symbol of order} $\rho \in\R$ if  $f \in C^\infty(\R^2, \C)$ and 
               $\forall \alpha, \beta \in \N_0$, there exists $C_{\alpha, \beta} >0$ such that
$$
 \vert \partial_x^\alpha \, \partial_\xi^\beta f( x,\xi)\vert \leq C_{\alpha,\beta} \ (1 + x^2 + \xi^2)^{\rho-\frac{\beta + \alpha}{2}}  \ . 
$$
             We will write $f \in S^\rho_{\hos}$.\\
$(ii)$ We shall say  that $f$ is a {\em classical symbol of order} 0,  $ f\in S^0_\cl$,  if there exists  $f_0 \in C^\infty(\R^2, \C)$  positively homogeneous of degree $0$, i.e. 
\begin{equation}
\label{hom.fun}
f_0( \lambda x , \lambda \xi) =  f_0(x, \xi) \ , \quad \forall \lambda \geq 1 , \ \ \  \forall \, x^2+\xi^2  \geq 1 ,
\end{equation}
and $\mu <0$ 
such that $f - f_0 \in S^{\mu}$. We shall call $f_0$ the {\em principal symbol} of $f$.
\end{definition}

\begin{remark}
$(i)$ It is easy to see that  $S^0_\cl \subset S^0$. \\
$(ii)$ With our numerology, the symbol $\frac12({x^2 + \xi^2})$ of the harmonic oscillator $H_0$ is of order 1,  and not of order 2 as typically in the literature.
\end{remark}

We shall also consider symbols depending periodically from time. 
We will denote by $C^k(\T, S^\rho)$, $k \in \N_0$, the space of $C^k$ maps $f\colon \T \ni t \mapsto f(t,\cdot) \in S^\rho$ with finite seminorms 
\begin{equation}
\label{seminorms}
\wp^{k,\rho}_j(f) := \sum_{\alpha + \beta \leq j \atop 0 \leq \ell \leq k } \ \ 
\sup_{x,\xi \in \R \atop t \in \T}
\frac{\left|\partial_x^\alpha \, \partial_\xi^\beta  \pa_t^\ell f(t, x,\xi)\right|}{ \left(1 + x^2 + \xi^2\right)^{\rho-\frac{\beta +\alpha}{2}} } \ , 
\qquad \forall j \in \N_0 \  . 
\end{equation}
Such seminorms turn $C^k(\T, S^\rho_\hos)$ into a Fr\'echet space, with  distance 
\begin{equation}
\label{dist}
\td^{k,\rho}(f, g):= \sum_{j \geq 0} \frac{1}{2^j}\,  \frac{\wp^{k,\rho}_j(f-g) }{1+\wp^{k,\rho}_j(f-g)} , \qquad \forall f, g \in C^k(\T, S^\rho_\hos) \ . 
\end{equation}
 Similarly we define the space $C^k(\T, S^0_\cl)$, which we endow with the  seminorms and distance in \eqref{seminorms}, \eqref{dist}.
   Finally we denote by
   $C^0_r(\T, S^0_\cl)$ the subset of $C^0(\T, S^0_{\cl})$ of real valued symbols.\\
 To a symbol $f \in C^k(\T,S^\rho)$ we associate the operator
$F(t,x, D)$  by standard Weyl quantization 
       $$
\big(F(t,x,D) \psi\big)(x):=   \Big(\Opw{f(t,\cdot)} \psi \Big)(x) := 
\frac{1}{2\pi} \iint_{y, \xi \in \R} {\rm e}^{\im (x-y)\xi} \, f\left(t, \frac{x+y}{2}, \xi \right) \, \psi(y) \, \di y \di \xi \ . 
   $$
   We shall say that an operator $F$ is a {\em pseudodifferential operator} of order $\rho$ if $F = \Opw{f}$ for some $f \in S^\rho$ and shall write $F \in \cS^\rho$. 
   If the symbol $f \in C^k(\T, S^\rho)$ we shall write $F \in C^k(\T, \cS^\rho)$.
   
\vspace{1em}

Our main result is the following one.

\begin{theorem}
\label{thm:main}
Denote by
\begin{equation}
  \label{sV}
  \begin{aligned}
  \sV:= \Big\lbrace 
  v \in C^0_r(\T, S^0_\cl) \colon  & \mbox{  the principal symbol } v_0 \mbox{ fulfills }   \{ \la v_0 \ra , \, h_0 \} \not\equiv 0  \mbox{ in } x^2 + \xi^2 \geq 1
  \Big\rbrace \   \ , 
  \end{aligned}
  \end{equation}
  where $\la v_0 \ra$ is the resonant average \eqref{res.av0} of $v_0$ and $h_0(x,\xi) = \frac12(x^2+\xi^2)$.
  Then:
\begin{itemize}
\item[(i)] For any  $v \in \sV$, the operator $V(t,x,D):=\Opw{v}$ is a  transporter for \eqref{har.osc}.
 More precisely, $\forall r > 0$  there exist  a
solution $\psi(t) \in \cH^r$ of \eqref{har.osc}
 and constants $C, T >0$ such that 
\begin{equation}
\label{thm:est}
\norm{\psi(t)}_r \geq C \la t \ra^r , \quad \forall t > T .
\end{equation}
\item[(ii)] The set $\sV$ is {\em generic} in $C^0_r(\T, S^0_\cl)$; precisely it is open and dense with respect to the metric $\td^{0,0}$ in \eqref{dist}.
\end{itemize}
\end{theorem}

Let us  comment the result.
\begin{enumerate}
\item 
The set  $\sV$ is strictly contained in $C^0_r(\T, S^0_\cl)$. For example, any symbol constant in time does not belong to $\sV$. Indeed if $v_0(x,\xi)$ is time independent, then its resonant average $\la v_0 \ra$  commutes with  $h_0$, as it is easily checked.\\
However,  given a time independent, non constant  symbol $v_0(x,\xi) \in S^0_{\cl}$, it is always possible to find $n \in \N$ so that 
$\cos(nt) v_0(x,\xi) \in \sV$, see Lemma \ref{lem:ex1}.

\item 
The open property of item $(ii)$ guarantees that transporters are stable under perturbations. 
In particular if $v \in \sV$, any  sufficiently small perturbation of $v$ still belongs to $\sV$.\\
The density instead  guarantees that given {\em any} symbol in $C^0_r(\T, S^0_\cl)$,   it is always possible to perturb it  so that the new symbol   belongs to $\sV$.

\item The property of belonging to $\sV$  involves only the principal symbol. In particular if $v \in \sV$, one can add arbitrarily large symbols in $C^0_r(\T, S^\rho)$, $\rho <0$, and still be in $\sV$.

 \item The growth of Sobolev norms of Theorem \ref{thm:main}  is truly an energy cascade  phenomenon;  indeed the $L^2$-norm of any solution of 
 \eqref{har.osc}  is preserved for all times, $\norm{u(t)}_{L^2}  = \norm{u(0)}_{L^2}$, $\, \forall t \in \R$. This is  due to the fact that $H_0 + V(t,x,D)$ is selfadjoint.

\item
Estimate \eqref{thm:est} is optimal, since it is proved in \cite{del} (see also \cite{MaRo}) that  {\em any} solution of \eqref{har.osc}  fulfills the upper bounds
$$
\forall r>0 \ \ \ \exists \,\wt  C_r > 0 \colon 
\quad 
\norm{\psi(t)}_r \leq \wt C_r \la t \ra^r \, \norm{\psi(0)}_r .
$$

\item Energy cascade is  a resonant phenomenon; here it happens because $V(t,x,D)$ oscillates at frequency $\omega = 1$  which resonates with the spectral gaps of the harmonic oscillator.  
In \cite{BGMR2} we proved  that if   $V(\omega t, x, D)$  is quasiperiodic in time  with a Diophantine frequency vector $\omega \in \R^n$, then the Sobolev norms of the solutions grow at most as  $\la t \ra^\epsilon$, $\forall \epsilon >0$ (see \cite{bambusi_langella} for recent results on $\la t\ra^\e$ growth and references therein). \\
Moreover, with additional restrictions on  $\omega$ (typically  belonging to some  Cantor set of large measure) and assuming $V(t,x,D)$ to be small in size, then all  solutions have uniformly in time bounded Sobolev norms \cite{Bam16I,BGMR1}. 
Therefore the stability/instability of the system depends strongly on the resonant properties of the frequency $\omega$.

\end{enumerate}

Let us briefly describe the main ingredients of the proof. 
The first step is to use resonant pseudodifferential normal form, analogous to the one of Delort \cite{del}, to conjugate the original equation \eqref{har.osc} to 
\begin{equation}
\label{intro.eff0}
\im \pa_t \vf = \big(\Opw{\la v_0\ra} + T + R(t) \big)\vf
\end{equation}
where $\la v_0\ra$ is the resonant average \eqref{res.av0} of the principal symbol of $V(t,x,D)$, $T$ is a selfadjoint, time independent pseudodifferential operator of negative order, and $R(t)$ is an arbitrary regularizing perturbation. This is done in Section \ref{sec:normalform}.

The second step is the analysis of the 
{\em effective} Hamiltonian
\begin{equation}
\label{intro.eff}
\im \pa_t \psi = \big(\Opw{\la v_0\ra} + T  \big)\psi
\end{equation}
obtained removing $R(t)$ from \eqref{intro.eff0}.
We construct solutions of \eqref{intro.eff} exhibiting dispersion in the energy space, i.e. solutions $\psi(t)$ whose negative $\cH^{-r}$-Sobolev norm, $r >0$,   decays in time at a polynomial rate:
\begin{equation}
\label{intro.decay}
\norm{\psi(t)}_{-r} \leq \frac{C}{\la t \ra^r} \norm{\psi(0)}_r , \quad \forall t \in \R \  .
\end{equation}
Hence, by the unitarity of the flow, these solutions have unbounded growth of positive Sobolev norms.
Then it is not difficult to construct solutions of the complete system \eqref{intro.eff0} exhibiting Sobolev norms explosion. 

To prove \eqref{intro.decay}  we study the spectral properties of $\fV_0:= \Opw{\la v_0\ra}$. 
Exploiting Mourre's commutator theory \cite{Mourre}, 
we show that 
 $\fV_0$ has absolutely continuous spectrum in an interval  $\cI \subset \R$,  over which it fulfills the strict Mourre estimate 
 \begin{equation}\label{intr.mourre}
g_\cI(\fV_0) \, \im [\fV_0, A ] \, g_\cI(\fV_0) \geq \rho \, g_\cI(\fV_0)^2
\end{equation}
for some $\rho >0$, $A \in \cS^1$ and  $g \in C^\infty_c(\R, \R_{\geq 0})$ with  $g \equiv 1$ on $\cI$.
This in turn implies, by Sigal-Soffer theory \cite{SS},  that
$\fV_0$   fulfills   
dispersive estimates in the frequency space  in the form of  local energy decay estimates.
The same is true also for  $\fV_0 + T$, as Mourre estimates are stable by  pseudodifferential operators of negative order.

The key difficulty in applying Mourre-Sigal-Soffer theory is that the operator $A$ entering the  Mourre estimate \eqref{intr.mourre} is not given, 
but  has to be constructed. 
In particular we produce $A$ so that the principal symbol of the operator on the left of \eqref{intr.mourre} is 
\begin{equation}
\label{intr.mourre2}
g_\cI(\la v_0\ra)^2  \,  \{ \la v_0 \ra, h_0 \}^2  \ .
\end{equation}
Then condition  \eqref{cond} allows to select  an interval  $\cI$ so that the function in  \eqref{intr.mourre2} is strictly positive, and then we deduce
\eqref{intr.mourre} exploiting the strong G\r{a}rding inequality.
 This is done in Section \ref{sec:mourre}. 

Let us compare our result with the previous ones in the literature.
After the pioneering work by Bourgain \cite{bourgain99}, 
Delort \cite{del} constructs the first example of a transporter for \eqref{har.osc}, which is also a classical pseudodifferential operator of order 0, and a solution whose $\cH^r$-norm grows as $t^r$.
 Bambusi-Gr\'ebert-M.-Robert \cite{BGMR1} constructs an unbounded transporter and a solution growing as $t^{2r}$. 
 The author \cite{Mas19} constructs a universal transporter, meaning a perturbation $V(t,x,D)$ such that {\em all} non trivial solutions of \eqref{har.osc} fulfill \eqref{limsup}.
 Faou-Raphael \cite{FaouRaph} deal with the very interesting case of a multiplication operator  $V(t,x)$ (and not a pseudodifferential operator) and construct a solution whose $\cH^r$-norm grows at a logarithmic speed.
 Thomann \cite{Thomann} constructs a transporter for the 2D Harmonic oscillator on the Bargmann-Fock space. 
 Finally Liang, Zhao and Zhou \cite{LZZ} and   Luo, Liang and Zhao \cite{LLZ} consider operators $V(\omega t,x,D)$ which are the quantization of polynomial symbols of order at most 2 and depend quasi-periodically in time with a frequency $\omega \in \R^d$. 
 They are able to completely describe the  quantum dynamics according to the  resonant properties of $\omega$, and even obtain solutions growing at the unusual speed\footnote{note that in \cite{LLZ} the space $\cH^s := D(H_0^{s/2})$, differently from \eqref{Hr}, so in our notation one has to substitute $s \leadsto 2s$ in \cite{LLZ}} of $t^{4s}$.
 Finally we mention Haus-M. \cite{HausMaspero} which considers   anharmonic oscillators.

Before closing this introduction, we mention that  constructing  solutions with unbounded orbits in nonlinear Schr\"odinger-like equations is very difficult.
Long time unstable orbits  have been constructed for the  
nonlinear Schr\"odinger equation on $\T^2$  \cite{CKSTT, guardia_kaloshin,
hani14,  haus_procesi15, guardia_haus_procesi16, GHMMZ, GG}, 
but truly  unbounded  orbits  are known only 
 for the cubic 
Szeg\H o equation on $\T$  \cite{gerard_grellier, gerard_grellier2} and 
the  cubic NLS  on $\R \times \T^2$ \cite{hani15}.

\section{Pseudodifferential operators}
In this section we collect some results about pseudodifferential operators in $\cS^\rho$.

\paragraph{Symbolic calculus.} This first class of results regards very basic properties of pseudodifferential operators and can be found in classical texts such as \cite{ho,Shubin}.

\begin{theorem}\label{thm:sym.cal}
 Let $a \in S^\rho_\hos$, $b \in S^\mu$, $\rho, \mu \in \R$. Then
\begin{itemize}
\item[(i)] {\em Action:} For any $s \in \R$, there are $C, M >0$ such that
$$
\norm{\Opw{a} u }_{s-\rho} \leq C_s \, \wp^\rho_{M}(a)	\, \norm{u}_s . 
$$
\item[(ii)] {\em Symbolic calculus:}  
One has $\Opw{a}^* = \Opw{\bar a}$.\\
There exists a symbol $c \in S^{\rho+\mu}$ such that $\Opw{a}\circ \Opw{b} = \Opw{c}$. Moreover $c - ab  \in S^{\rho+\mu-1}$ and 
$$
\forall j \in \N_0 , \qquad \exists N, C >0  \mbox{ s.t. }  \ \ 
\wp^{\rho+\mu-1}_j(c - ab ) \leq C \, \wp^{\rho}_N(a) \, \wp^{\mu}_N(b) \ . 
$$ 
 There exists a symbol $d \in S^{\rho+\mu-1}$ such that 
 $\im [\Opw{a},  \Opw{b}] = \Opw{d}$. 
 Moreover $  d -  \{a , b \} \in S^{\rho + \mu -3} $ and 
  $$
\forall j \in \N_0 , \qquad \exists N, C >0  \mbox{ s.t. }  \ \ 
\wp^{\rho+\mu-3}_j( d -  \{a , b \} ) \leq C \, \wp^{\rho}_N(a) \, \wp^{\mu}_N(b) \ . 
$$ 
\item[(iii)]{\em Exact Egorov theorem:} One has
$$
e^{\im \tau H_0}\, \Opw{a} \, e^{-\im \tau H_0}
= \Opw{a\circ \phi^\tau} \ ,\quad \forall \tau \in \R  
$$
where $\phi^\tau$ is the flow \eqref{flow.ham}.
In particular $t \mapsto a\circ \phi^t \in  C^k(\T, S^\rho)$, $\forall k \in \N_0$,  and
  \begin{equation}
  \label{est.4}
  \forall  j\in \N_0  \quad \exists N, C >0  \ s.t.\  \   \wp^{\rho}_j(a\circ \phi^t) \leq C \, \wp^\rho_N(a)  \ . 
  \end{equation}
\item[(iv)] {\em Compactness:} Let $\rho <0$. Then $\Opw{a} $ is compact. 
\end{itemize}
\end{theorem}

\begin{remark}\label{rem:res.av}
If $v \in C^0(\T, S^\rho)$, $\rho \in \R$, then its resonant average $\la v \ra(x,\xi)$ (defined in \eqref{res.av0}) belongs to $S^\rho$, as one checks using the explicit expression
\begin{equation}\label{symb}
\la v \ra(x,\xi) := 
\frac{1}{2\pi} \int_0^{2\pi}v(t,x \cos t + \xi \sin t, -x \sin t + \xi \cos t) \, \di t  \ . 
\end{equation}
Note also that  $\la v \ra$ is time-independent.
If $v$ is real valued, so is $\la v \ra$.
\end{remark}

\paragraph{Flows.} The second class of results regards the flow generated by  pseudodifferential operators.  For the proofs we refer to  \cite{BGMR2}.
\begin{lemma}
\label{MR}
 Assume that $X \in  C^1(\T, \cS^1)$ is selfadjoint.
 Then the following holds true.
 \begin{itemize}
 \item[(i)] {\em Flow:} $e^{- \im \tau X(t)}$ extends to an operator in $\cL(\cH^r)$   $\, \forall r \in \R$, and moreover there exist $c_r, C_r >0$ s.t.
 \begin{equation}
 \label{X(t).est0}
 c_r \norm{\psi}_r \leq \norm{e^{-\im \tau X(t)} \psi}_r \leq 
 C_r \norm{\psi}_r \ , \qquad \forall t \in \R \ , \quad \forall \tau \in [0,1]  \ .
 \end{equation}
 \item[(ii)] {\em Conjugation:} Let $H(t)$ be a time dependent selfadjoint operator. Assume that
$\psi(t)=e^{-\im X(t)}\vf(t)$ then
\begin{equation}
\label{1}
\im\pa_t \psi=H(t) \psi\ \quad \iff \quad \im\pa_t \vf  =  H^+(t) \vf 
\end{equation}  
where 
\begin{align}
\label{4.1.1}
 H^+(t):= e^{\im X(t)} \, H(t)\, e^{-\im X(t)}
-\int_0^1 e^{\im s X(t) } \,(\pa_t X(t))  \, e^{-\im
  s X(t) }  \ \di s \ .
\end{align}
\item[(iii)]  {\em Lie expansion:} Let $X  \in  \cS^\rho$ with  $\rho <1$ and  $H\in  \cS^{m}$, 
$m \in \R$, both selfadjoint. 
Then $e^{\im X} \, H\, e^{-\im  X} \in \cS^m$, it is selfadjoint and for any $M\geq 1$ we have
\begin{equation}
\label{expansion}
e^{\im  X}\, H \,  e^{-\im X } = \sum_{\ell=0}^{M}\frac{\im^\ell}{ \ell!}{\rm ad}_{X}^\ell(H) + R_M ,
\end{equation}
where ${\rm ad}_{X}(H):= [X,H]$ and 
$$
R_M :=  \frac{\im^{M+1}}{M!} \int_0^1 (1-\tau)^M
\, e^{\im \tau X}\, {\rm ad}_{X}^{M+1}(H) \,  e^{-\im \tau X }  \, \di \tau  \  \in \cS^{{m -(M+1)(1-\rho)}}
$$
is selfadjoint.
\end{itemize}
\end{lemma}

\begin{remark}
If $X\in C^{k}(\T, \cS^\rho)$, $H \in C^k(\T, \cS^m)$ with $\rho <1$ and $m \in \R$, then the remainder $R_M(t)$ in formula \eqref{expansion} belongs to
$C^k(\T, \cS^{{m -(M+1)(1-\rho)}})$.

\end{remark}

\paragraph{Functional calculus.} The next results concern the functional calculus of pseudodifferential operators in $\cS^\rho$. Standard references are the papers 
\cite{HS,davies} and the books \cite{robook, ABG}. 

Given $g\in C^\infty_c(\R, \R)$, we define its {\em almost analytic extension} as follows: for any  $N \in \N$, put	
	\begin{equation}
	\label{}
\wt g_N \colon \R^2 \to \C , \qquad 	\wt g_N(x, y)  :=  \left( \sum_{\ell = 0}^N g^{(\ell)}(x) \frac{(\im y)^\ell}{\ell !} \right) \, \tau \left(\frac{y}{\la x \ra} \right)
	\end{equation}
	where $\tau \in C^\infty(\R, \R_{\geq 0})$ is a  cut-off function fulfilling $\tau(s) = 1$ for $|s| \leq 1$ and $\tau(s) = 0$ for $|s| \geq 2$.
Then, given 	$\fH$ a selfadjoint operator, one  defines  $g(\fH)$ via  the {\em Helffer-Sj\"ostrand formula} 
\begin{equation}
	\label{f(T)}
	g(\fH)  := 
	  - \frac{1}{\pi} \int_{\R^2} \frac{\pa \wt g_N(x,y)}{\pa {\bar z}} \, (x+\im y-\fH)^{-1} \di x \, \di y  \ , 
	  \qquad
	 \frac{\pa }{\pa {\bar z}} := \frac{1}{2}\left(  \frac{\pa }{\pa x}  + \im  \frac{\pa }{\pa y} \right) \ .
	\end{equation}	
The operator $g(\fH)$ is independent of $N \in \N$, see \cite{davies}.

\begin{theorem}[Functional calculus]\label{lem:fc}
 Let $g \in C^\infty_c(\R, \R)$. 
\begin{itemize}
\item[(i)] Let $v \in S^0$ be real valued. The operator $g(\Opw{v})$, defined via \eqref{f(T)}, belongs to  $\cS^0$. 
Moreover 
$ g(\Opw{v})- \Opw{g(v)} \in  \cS^{-1} $. 
\item[(ii)] Let $\fA, \fB \in \cL(\cH)$ be selfadjoint. If  $\fA - \fB$ is compact, so is $g(\fA) - g(\fB)$.
\end{itemize}
\end{theorem}

\paragraph{Strong G\r{a}rding inequality.} The next result  is the strong G\r{a}rding inequality for symbols in $S^0$.

\begin{theorem}[Strong G\r{a}rding inequality] 
\label{thm:garding}
Let $a \in S^0$ and assume there exists $R >0$ such that
\begin{equation}
\label{garding.a.ass}
a(x,\xi) \geq 0 , \quad \forall x^2 + \xi^2 \geq R  \ . 
\end{equation}
Then there exists $C >0$ such that
\begin{equation}
\label{sGarding}
\la \Opw{a} u , u \ra \geq - C \norm{u }_{-1}^2 , \quad \forall u \in L^2 \ .
\end{equation}
\end{theorem}
This result is probably well known but   we couldn't find a proof of this exact statement  in the literature, so we prove it in  Appendix \ref{app:Garding}.

\paragraph{Essential spectrum.}
In this last paragraph we characterize  the essential spectrum of  pseudodifferential operators with symbols in $ S^0$. 
\begin{theorem}\label{thm:spectrum_ho}
Let $v \in S^0_\hos$ be real valued. Then
\begin{equation}
\label{spectrum}
\sigma_{ess}(\Opw{v}) = 
\left\lbrace
\lambda \in \R \colon \ \ 
\exists \, (x_j, \xi_j) \to \infty  \ \ 
\mbox{ s.t. }  \ \ 
v(x_j, \xi_j) \to \lambda 
\right\rbrace \ . 
\end{equation}
\end{theorem}
We prove  this result in Appendix \ref{app:spectrum}.

\section{Pseudodifferential normal form}\label{sec:normalform}
The first step of our proof is to use  pseudodifferential normal form to extract from the original equation \eqref{har.osc} an effective Hamiltonian having as  leading term the operator $\Opw{\la v_0 \ra}$.
Precisely we shall prove the following result.	
	
	\begin{proposition}
\label{prop:rpdnf}
Consider equation \eqref{har.osc} with $V(t,x,D) = \Opw{v(t,\cdot)}$ and $v  \in C^0_r(\T, S^0_{cl})$. 
There exists $\delta >0$ such that for any $ N \in \N$ the following holds true. 
 There  exists a change of coordinates 
 $\cU_N(t)$,  unitary in $L^2$ and  fulfilling
\begin{equation}
\label{unitary}
\forall r \geq 0  \quad \exists \, c_r, C_r >0 \colon \qquad c_r \norm{\vf}_r  \leq \norm{\cU_N(t)^{\pm} \vf}_r \leq C_r \norm{\vf}_r , \qquad \forall t \in \R , 
\end{equation}
 such that  $u(t)$ solves  \eqref{har.osc} if and only if 
 $\vf(t) := \cU_N(t) u(t)$ solves 
\begin{equation}
\label{res.eq}
\im \pa_t \vf = \big(\Opw{\la v_0 \ra} +  T_N + R_N(t) \big)\vf 
\end{equation}
where $\la v_0 \ra$ is the resonant average of the principal symbol  $v_0$ of $v$, 
  $T_N \in \cS^{-\delta}$ is  time independent and selfadjoint and 
 $R_N \in C^0(\T, \cS^{-N-\delta})$ is selfadjoint $\forall t$.
\end{proposition}

The proof of the proposition is based on a series of change of coordinates, and at each step we shall solve a simple homological equation, which we now describe:

\begin{lemma}\label{lem:hom}
Assume that $H \in C^0(\T, \cS^m)$, $m \in \R$. There exists $X \in C^1(\T, \cS^m)$ such that
\begin{equation}
\label{eq:hom0}
H(t) - \frac{1}{2\pi}\int_0^{2\pi} H(s) \, \di s = \pa_t X(t) \ .
\end{equation}
If $H(t)$ is selfadjoint for any $t$, so is $X(t)$.
\end{lemma}	
\begin{proof}
We solve   \eqref{eq:hom0} at  the level of symbols; we put   $H(t) = \Opw{h(t, \cdot)}$ with   $h \in C^0(\T, S^m)$ and look for 
$X(t) = \Opw{\chi(t, \cdot)}$ with a symbol $\chi \in C^1(\T, S^m)$ to be determined.
Then \eqref{eq:hom0} reads
\begin{equation}
\label{hom.5}
h(t,x, \xi) - \frac{1}{2\pi} \int_0^{2\pi} h(s, x, \xi) \, \di s =  \pa_t \chi(t,x,\xi) \ , 
\end{equation} 
whose  solution is readily given by 
\begin{equation}
\label{hom.6}
\chi(t,x, \xi) := \int_0^{t} \left(   h(s, x,\xi) -  \frac{1}{2\pi} \int_0^{2\pi} h(s_1, x, \xi) \, \di s_1 \right)  \, \di s \ . 
\end{equation}
One verifies that $\chi  \in C^1(\T,S^m)$.
Finally if $H(t)$ is selfadjoint, its symbol $h(t,\cdot)$ is real valued, so is $\chi(t,\cdot)$ and thus  $X(t)$ is selfadjoint.
\end{proof}

\begin{proof}[Proof of Proposition \ref{prop:rpdnf}]
First we gauge away $H_0$, performing the change of variables
$u= e^{- \im t H_0} \psi$. Then $\psi$ solves 
$\im \pa_t \psi =  e^{\im t H_0} \, \Opw{v(t,\cdot)} \, e^{- \im t H_0}  \psi $.\\
Next recall that, 
by definition of the class $S^0_\cl$, the symbol $v$ decomposes as 
\begin{equation}
\label{v.dc}
v(t,\cdot) = v_0(t,\cdot)  + v_{-\mu}(t, \cdot) 
\end{equation}
with  $v_0(t, \cdot)$ positively homogeneous of degree 0 and $v_{-\mu} \in C^0_r(\T, S^{-\mu})$ for some $\mu >0$. 
So we decompose
\begin{equation}
\label{Ht}
\im \pa_t \psi = \left( \fV(t) + R(t) \right)  \psi 
\end{equation}
with
\begin{equation}
\label{Ht2}
\fV(t):= e^{\im t H_0} \, \Opw{v_0(t,\cdot)} \, e^{- \im t H_0} \ , 
\qquad
R(t) :=  e^{\im t H_0} \, \Opw{v_{-\mu}(t,\cdot)} \, e^{- \im t H_0} \ . 
\end{equation}
Note that $ \fV  \in C^0(\T, \cS^0)$  and $R \in C^0(\T, \cS^{-\mu})$, both selfadjoint $\forall t$.

Now we proceed inductively,  performing several changes of coordinates which remove the oscillating part  order by order.
Let us describe the first step.
We perform the  change of variables $\psi= e^{ -\im X_1(t)}
\vf$ where $X_1(t)\in C^1(\T, \cS^{0 })$,   selfadjoint  $\forall t$,
to be determined.
By Lemma \ref{MR},  $\vf$ fulfills the Schr\"odinger equation
$\im \pa_t \vf = \fH_1(t) \vf$ with
\begin{align*}
\fH_1(t)&:= e^{\im X_1(t)} \, (\fV(t)+ R(t)) \, e^{-\im X_1(t)} 
-\int_0^1 e^{\im s X_1(t) } \,(\pa_t X_1(t) ) \, e^{-\im
  s X_1(t) }  \ \di s  \ . 
  \end{align*}
  Then a Lie  expansion, see  \eqref{expansion}, gives
  \begin{align}
  \label{4.1.2}
  \fH_1(t)
  &= \fV(t)   - \pa_t  X_1(t) +    R_1(t) , \qquad
  R_{1}  \in  C^0(\T, \cS^{-\delta}), \mbox{  selfadjoint } \forall t  \ ,
\end{align} 
with \begin{equation}
\label{}
\delta := \min (1,\mu) >0 \ . 
\end{equation}
Lemma \ref{lem:hom} allows us to  choose     $X_1 \in C^1(\T, \cS^0)$, selfadjoint  $\forall t$, s.t.
\begin{equation}
\label{hom.4}
\fV(t)  - \frac{1}{2\pi}\int_0^{2\pi} \fV(s) \, \di s =  \pa_t X_1(t)  \ . 
\end{equation}
With this choice, and exploiting Egorov's Theorem, equation \eqref{4.1.2} reduces to 
  \begin{align}
\notag
\fH_1(t)&=\frac{1}{2\pi}\int_0^{2\pi} \fV(s)\, \di s + R_1(t)   
 =    \frac{1}{2\pi} \int_0^{2\pi}   \Opw{v_0(s,\phi^s(x,\xi))}   \, \di s + R_1(t) \\
 \label{4.1.3}
& =  \Opw{\la v_0\ra} + R_1(t) \ . 
\end{align} 
Note also that the map $e^{- \im X_1(t)}$ fulfills an estimate like  \eqref{unitary} thanks to Lemma \ref{MR}. \\
Then we iterate this procedure. Assume that at the $N$-th step we have the equation
\begin{equation}\label{ps0}
\im \pa_t \vf =\left(  \Opw{\la v_0\ra} + T_N + R_N(t)\right) \vf   =: \fH_N(t) \vf
\end{equation}
with $T_N \in \cS^{-\delta}$ selfadjoint and time independent, and $ R_N \in C^0(\T, \cS^{-N-\delta})$ selfadjoint $\forall t$.
The change of variables $\vf = e^{- \im X_{N+1}(t)} w$, with a certain $X_{N+1} \in C^1(\T, \cS^{-N-\delta})$ selfadjoint and to be determined, conjugates
\eqref{ps0} to $ \im \pa_t w = \fH_{N+1}(t) w$ with 
\begin{equation}
 \label{hom.7} 
 \fH_{N+1}(t) =\Opw{\la v_0\ra} + T_N + R_N(t)- \pa_t X_{N+1}(t) + R_{N+1}(t)  \ , 
 \end{equation} 
with $  R_{N+1} \in C^0(\T, \cS^{-(N+1)-\delta}) $ selfadjoint $\forall t$.
We exploit  again Lemma \ref{lem:hom} to find $X_{N+1}$ with the wanted properties solving
 $$
 \pa_t X_{N+1}  = R_{N}(t) - \frac{1}{2\pi}\int_0^{2\pi} R_{N}(s) \di s \ . 
 $$
  Thus $\fH_{N+1}(t)$ in \eqref{hom.7} becomes
 \begin{align*}
 \label{hom.8}
\fH_{N+1}(t)
& = \Opw{\la v_0\ra}  + T_{N+1}  + R_{N+1}(t)  , \qquad
T_{N+1}:= T_N +  \frac{1}{2\pi}\int_0^{2\pi} R_{N}(s)\, \di s  \ .
 \end{align*}
 Clearly $T_{N+1} \in \cS^{-\delta}$, it is 
 selfadjoint and time-independent. 
This proves the inductive step.

Finally we put $\cU_N(t) :=   e^{\im  X_N(t)} \, \cdots e^{\im  X_1(t)}e^{\im t  H_0}$; estimate \eqref{unitary} follows from Lemma \ref{MR}.
\end{proof}

\section{Mourre estimates}\label{sec:mourre}
In this section we study the spectral properties of the operator
$\Opw{\la v_0 \ra}$, which is the  leading term of equation \eqref{res.eq}.
In particular we prove that its essential spectrum contains a nontrivial interval $\cI_0$ and, over a subinterval $\cI\subset \cI_0$, it fulfills a Mourre estimate.
Then Mourre theory \cite{Mourre} guarantees that $\Opw{\la v_0 \ra}$ has  absolutely continuous spectrum in $\cI$, and 
therefore we might expect  dispersive behavior.
 More precisely we  prove the following result:
\begin{proposition}
\label{lem:har.mourre2}
Let $v \in \sV$ (see \eqref{sV}) and $v_0$ be its principal symbol. 
Denote by $\la v_0 \ra \in S^0$ its resonant average (see \eqref{res.av0}) and
$\fV_0 := \Opw{\la v_0 \ra }. $
The following holds true.
\begin{itemize}
\item[(i)] The spectrum  $ \sigma\big(\fV_0\big)$ contains a closed interval $\cI_0$.
\item[(ii)] There exist $A \in \cS^1$ selfadjoint, an open  interval $\cI \subset \cI_0$,  a function  $g_{\cI} \in C^\infty_c(\R, \R_{\geq 0})$ with $g_{\cI} \equiv 1$ over $\cI$
such that  
\begin{equation}
\label{mourre.ver}
g_{\cI}\big(\fV_0\big) \ \im [\fV_0, A] \  g_{\cI}\big(\fV_0\big) \geq \rho \,  g_{\cI}\big(\fV_0\big)^2  + \fK  
\end{equation}
for some  $\rho >0$  and $\fK$ a compact operator.
\end{itemize}

\end{proposition}

{\bf Notation:} 
In the following, we shall denote by $\fK$ a general compact operator on $L^2$, which might change from line to line. 
We shall also constantly  use  that pseudodifferential operators in $\cS^{-\mu}$, $\mu>0$, are compact on $L^2$.

\vspace{1em}
In order to prove Proposition \ref{lem:har.mourre2} we start with some  preparation. We shall denote by
 $\cA \colon  \T \times \R_{>0} \to \R^2\setminus\{0\}$, $(\vartheta, I) \to (x,\xi)$ the action-angle change of coordinates defined by
  \begin{equation}
  \label{aa}
x = \sqrt{2I} \sin(\vartheta), \ \ \xi = \sqrt{2I} \cos(\vartheta), \qquad \forall (x, \xi) \neq (0,0)  \ .
  \end{equation}
Given a function $f \in C^\infty(\R^2, \C)$, we shall denote by
 $\wt f:= f \circ \cA$ the function $f$ expressed in action angle coordinates \eqref{aa}. 
 Note that since  \eqref{aa} is a canonical transformation, one has
 \begin{equation}
 \label{pp.aa}
\wt{ \{ f, g \}} = \{ \wt f, \wt g \}=: \pa_I \wt f \, \pa_\vartheta \wt g - \pa_\vartheta \wt f \, \pa_I \wt g \ . 
 \end{equation}
  \begin{lemma}\label{lem:aa} 
  Assume that $f \in C^\infty(\R^2, \C)$ is positively homogeneous of degree 0, i.e. it fulfills \eqref{hom.fun}. 
Then   
  \begin{equation}
  \label{wtv.angle}
 \wt f(\vartheta, I) =  \wt f(\vartheta, 1) , \quad \forall I > \frac12 \ .
  \end{equation}
  In particular $\pa_I  \wt f(\vartheta, I)  = 0$ for any $I > \frac12$.
  \end{lemma}
  \begin{proof}
 Being $f$  positively homogeneous of degree 0, 
    \begin{equation}
  \wt f(\vartheta, I) = f( \sqrt{2I} \sin(\vartheta),  \sqrt{2I} \cos(\vartheta) )
  =
 f( \sin(\vartheta), \cos(\vartheta)  )
  =  
   \wt f(\vartheta, \frac12) \  ,  \quad  \forall I \geq \frac12 \ ,
  \end{equation}
implying in particular the identity \eqref{wtv.angle}.
  \end{proof}
Now we begin  the  proof of  Proposition 
\ref{lem:har.mourre2}, which is 
 splitted in several lemmas. 
 \begin{lemma}
 Let $v \in C^0_r(\T, S^0_{cl})$ and denote by  $v_0$ its principal symbol. Then
 \begin{itemize}
 \item[(i)] $\la v_0\ra$ is real valued and  positively homogeneous of degree 0, i.e.
\begin{equation}
\label{v0-hom}
\la v_0 \ra (\lambda x, \lambda \xi) = \la v_0 \ra ( x,  \xi)  , \quad \forall \lambda \geq 1 , \ \ \ \forall x^2 + \xi^2 \geq 1 \ . 
\end{equation}
 \item[(ii)] If furthermore $v \in \sV$, then 
\begin{equation}
\label{I0}
 \cI_0 := \left[ \min_{{x^2 + \xi^2} = 1} \la v_0\ra(x, \xi), \, \max_{x^2 + \xi^2 = 1}  \la v_0\ra(x,\xi)   \right] \subseteq \sigma_{ess}\big(\fV_0\big)
\end{equation}
 and it has nonempty interior.
 \end{itemize}
 \end{lemma}
 \begin{proof}
 $(i)$
 It follows from the explicit expression \eqref{symb} using that  $v_0(t,\cdot)$ is real valued and  positively homogeneous of degree 0. \\
$(ii)$ 
By Theorem \ref{thm:spectrum_ho}
\begin{align*}
 \sigma_{ess}\big(\fV_0\big)
&= \left\lbrace \lambda \in \R \colon   \exists\,  (x_j, \xi_j) \to \infty \mbox{ s.t. } \la v_0 \ra (x_j, \xi_j) \to \lambda  \right\rbrace . 
\end{align*}
Since $\la v_0\ra$ is positively homogeneous of degree 0 and definitively one has  $x_j^2 + \xi_j^2 \geq 1$, we have that 
$$
\la v_0\ra (x_j, \xi_j) = \la v_0 \ra ( \frac{x_j}{x_j^2 + \xi_j^2}, \frac{\xi_j}{x_j^2 + \xi_j^2}) \ ,
$$
and we deduce that 
\begin{align*}
 \sigma_{ess}\big(\fV_0\big)
& =  \left\lbrace \lambda \in \R \colon   \exists\,  \{(\wt x_j, \wt \xi_j)\}_{j\geq 1}\mbox{ with } {\wt x}_j^2 + {\wt \xi}_j^2 = 1 \ 
 \forall j \ \ \  \mbox{ s.t. } \la v_0 \ra(\wt x_j, \wt \xi_j) \to \lambda  \right\rbrace \ .
 \end{align*}
Finally, since the function $\la v_0 \ra$ is smooth, this set corresponds to the image of $\la v_0 \ra\vert_{x^2+\xi^2 = 1}$ and therefore 
\begin{align*}
 \sigma_{ess}(\fV_0)
& =[ \min_{x^2 + \xi^2 = 1} \la v_0\ra(x, \xi), \, \max_{x^2 + \xi^2 = 1}  \la v_0\ra(x,\xi)   ] = \cI_0 .
\end{align*}
To prove  that  $\cI_0$ is a nontrivial  interval, it is sufficient to show that $\la v_0\ra$ is not constant on $x^2+\xi^2 = 1$. 
Assume by contradiction that $\la v_0\ra$ is constant on the circle; then  by homogeneity, it is constant on the whole set $x^2 + \xi^2 \geq 1$. 
But then 
$\{ h_0, \la v_0 \ra \} \equiv 0$ in $x^2 + \xi^2 \geq 1$, contradicting the assumption that $v \in \sV$.
\end{proof}

Next define the operator 
\begin{equation}\label{def:A}
A := \Opw{a} , \qquad a :=  \{ \la v_0 \ra , h_0 \}   \, h_0  \in S^1 \ , \ \ \mbox{real valued}.
\end{equation}
Clearly $A \in \cS^1$ and it is selfadjoint. 

\begin{lemma}\label{lem:Vmourre1}
Let   $A$ be defined in \eqref{def:A}. For any $\cI \subset \cI_0$ and $g_{\cI} \in C^\infty_c(\R, \R_{\geq 0})$ with $g_{\cI} \equiv 1$ over $\cI$  there exists a compact operator $\fK$ such that 
 \begin{equation}
 \label{gVAg}
g_{\cI}(\fV_0)  \, \im [\fV_0, A] \, g_{\cI}(\fV_0) = \Opw{g_\cI(\la v_0\ra)^2\, \{\la v_0\ra ,h_0\}^2} + \fK \ .
 \end{equation}
\end{lemma}

\begin{proof}
By  symbolic calculus we get
\begin{equation}
\label{eq:V.A.comm}
\im [ \fV_0 , A ] = \Opw{\{ \la v_0 \ra, a\}+ r_{-2}}, \quad
r_{-2} \in S^{-2},  \mbox{ real valued } \ . 
\end{equation}
Consider now the operator $g_{\cI}( \fV_0)$. Using  Lemma  \ref{lem:fc} we write 
\begin{align}
\label{g.dec}
g_{\cI}(\fV_0) & =  
\Opw{g_\cI(\la v_0 \ra)} + \fK
\end{align}
with $\fK$ a  compact operator.
We use the expressions \eqref{eq:V.A.comm} and \eqref{g.dec}, the fact that  
$\{ \la v_0 \ra, a\}$, $g_\cI(\la v_0 \ra)$ belong to $S^0$
and symbolic calculus  to
finally get 
\begin{align}
\label{g.dec2}
g_{\cI}( \fV_0) \, \im [\fV_0, A] \, g_{\cI}( \fV_0) 
& =  \Opw{g_\cI(\la v_0\ra)^2\, \{\la v_0\ra ,a \}} + \fK
\end{align}
with $\fK$  compact.
Next we compute $\{ \la v_0 \ra, a\}$. Using \eqref{def:A},
\begin{equation}
\label{v0.a}
\{ \la v_0 \ra, a\} =  \{ \la v_0 \ra, h_0\}^2 + w, \qquad w:= 
\{\la v_0 \ra,  \, \{ \la v_0\ra , h_0\} \} h_0 . 
\end{equation}
We claim that $w \in C^\infty_c(\R^2, \R) \subset S^{-\infty}$; then   \eqref{gVAg} follows by inserting this decomposition  in  \eqref{g.dec2}.
To prove  that $w \in C^\infty_c(\R^2, \R)$
we pass to action-angle variables defined in \eqref{aa}. As 
\begin{equation}
\label{pp.ac}
 \wt h_0 = I \ , \quad
 \wt{\{ \la v_0\ra, h_0\}} = - \pa_\vartheta \wt{ \la v_0\ra}(\vartheta, I) \ , 
\end{equation}
we get
$$
\wt{\{ \la v_0 \ra,  \, \{ \la v_0\ra, h_0\}\}} = \{ \wt{\la v_0 \ra}, \, \{\wt {\la v_0 \ra}, I \} \} 
\stackrel{\eqref{pp.aa}}{=} -\frac{\pa  \wt{\la v_0 \ra}}{\pa I} \, \frac{\pa^2  \wt{\la v_0 \ra}}{\pa \vartheta^2} +
\frac{\pa   \wt{\la v_0 \ra}}{\pa \vartheta} \, \frac{\pa^2  \wt{\la v_0 \ra}}{\pa I \pa \vartheta } \ . 
$$
Since $\la v_0 \ra$ is positively homogeneous of degree 0 (see \eqref{v0-hom}), by Lemma \ref{lem:aa} we have that
$\pa_I  \wt{\la v_0 \ra}(\vartheta, I)  = \pa^2_{I \vartheta}  \wt{\la v_0 \ra}(\vartheta, I) = 0$ for any $I > \frac12$, proving 
that $\wt{\{ \la v_0 \ra,  \, \{ \la v_0\ra, h_0\}\}} \in C^\infty_c(\R^2\setminus \{0\}, \R)$. Hence also $w \in C^\infty_c(\R^2, \R)$.
\end{proof}
The next one is the most important lemma, which proves the commutator estimate.
\begin{lemma}\label{lem:Vmourre2}
There exist an interval $\cI \subset \cI_0$, a function $g_\cI\in C^\infty_c(\R, \R_{\geq 0})$ with $\textup{supp } g_\cI \subset \cI_0$, $ g_\cI\equiv 1$ over $\cI$ and  numbers $\rho, C >0$  such that
\begin{equation}
\label{quant.sign}
\la \Opw{g_\cI(\la v_0\ra)^2\, \{\la v_0\ra,h_0\}^2} u, u \ra \geq \rho \,  \la \Opw{g_\cI(\la v_0\ra)^2}  u, u \ra - C \norm{u}_{-1}^2,  \qquad \forall u \in L^2   \ . 
\end{equation}
\end{lemma}
\begin{proof}
We claim the existence of  $\cI \subset \cI_0$  and $\rho >0$ such that
\begin{equation}
\label{sign}
g_\cI(\la v_0\ra)^2\, \{\la v_0\ra,h_0\}^2 \geq \rho  \,  g_\cI(\la v_0\ra)^2  , \quad \forall \, x^2 + \xi^2 \geq  1 . 
\end{equation}
Then the Strong G\r{a}rding inequality in Theorem \ref{sGarding} gives \eqref{quant.sign}. \\
To prove \eqref{sign} we pass to action-angle variables \eqref{aa}.
Using \eqref{pp.ac} 
 the left-hand-side of \eqref{sign} reads
$$
g_\cI\left(\widetilde{\la v_0 \ra} (\vartheta, I) \right)^2\,  \left[  \pa_\vartheta \widetilde{\la v_0 \ra} (\vartheta, I) \right]^2  \ . 
$$
By Lemma \ref{lem:aa},  
 $\widetilde{\la v_0 \ra}(\vartheta,I)\equiv  \widetilde{\la v_0 \ra}(\vartheta,1)$ for any $I >\frac12$.
Moreover, as  $ \widetilde{\la v_0 \ra}(\vartheta):=  \widetilde{\la v_0 \ra}(\vartheta,1)$
 is a smooth function defined on $\T$, 
Sard's theorem implies that the set $\cC:= \{\vartheta \in \T \colon \pa_\vartheta\widetilde{\la v_0 \ra}(\vartheta) = 0 \}$ of critical values of 
$\widetilde{\la v_0 \ra}$ has image $\widetilde{\la v_0 \ra}(\cC)$ of zero measure. 
Since the image of $\widetilde{\la v_0 \ra}(\cdot) \equiv 
\la v_0\ra\vert_{x^2+\xi^2 = 2} = 
\la v_0\ra\vert_{x^2+\xi^2 = 1}$ is the nontrivial interval $\cI_0$  (see \eqref{I0}), 
$\cI_0 \setminus \widetilde{\la v_0 \ra}(\cC)$ is a positive measure set and contains only  regular values.
 So fix    $\und\lambda \in \cI_0 \setminus \widetilde{\la v_0 \ra}(\cC)$.
The set  $\widetilde{\la v_0 \ra}^{-1}(\und\lambda)$ is a compact set of isolated points, thus finitely many;  denote them by  $\{\vartheta_a\}_{a=1}^d$. 
 Since $\pa_\vartheta \widetilde{\la v_0 \ra}(\vartheta_a) \neq 0$ for any $ a= 1, \ldots, d$, we
  can find 
neighbors $ \cU_a \subset \T$ of $\vartheta_a$ and a neighbor $\cV \subset \cI_0$ of $\und \lambda$ such that 
\begin{itemize}
\item[(i)] $ \widetilde{\la v_0 \ra}\colon \cU_a \to \cV$ is a diffeomorphism for any $a = 1, \ldots, d$,
\item[(ii)]  $ \widetilde{\la v_0 \ra}^{-1}(\cV) = \cup_a \cU_a$,
\item[(iii)] there exists $\rho >0$ such that $\min\limits_{\vartheta \in \cup_a \bar \cU_a } \abs{ \pa_\vartheta \widetilde{\la v_0 \ra}(\vartheta)} \geq \sqrt{ \rho} $ . 
\end{itemize}
Now take an interval $\cI \subset \cV$ with $\und \lambda \in \cI$. 
Take also $g_\cI \in C^\infty_c(\R,\R_{\geq 0})$ with $g_\cI \equiv 1$ on $\cI$ and $\textup{supp } g_\cI \subset \cV$.
Using (ii) above, we have that
\begin{equation}
\label{}
(\vartheta,I) \in \textup{supp }\left( g_\cI \big(\widetilde{\la v_0 \ra}(\vartheta,I)  \big)
\right)  \cap \{ I > \frac12  \} \ \ \Rightarrow \ \ \vartheta \in  \bar{ \bigcup_{a = 1}^d \cU_a } \ .
\end{equation}
In particular in the set
$ \textup{supp}\left( g_\cI \big(\widetilde{\la v_0 \ra}(\vartheta,I)  \big)
\right)  \cap \{ I > \frac12  \}$ 
 the function $\widetilde{\la v_0 \ra} (\vartheta, I) \equiv  \widetilde{\la v_0 \ra} (\vartheta)$ fulfills (iii) above.
We 
 deduce that 
$$
g_\cI^2\left(\widetilde{\la v_0 \ra} (\vartheta, I) \right)\,  \left(  \pa_\vartheta \widetilde{\la v_0 \ra} (\vartheta, I) \right)^2 \,  \geq \rho \ g_\cI^2\left(\widetilde{\la v_0 \ra} (\vartheta, I) \right)   \ , 
\quad
\forall (\vartheta, I) \in \T \times \{ I > \frac12 \} \ , 
$$
proving \eqref{sign}.
\end{proof}

We can finally prove Proposition  \ref{lem:har.mourre2}.

\begin{proof}[Proof of Proposition \ref{lem:har.mourre2}]
By Lemmata \ref{lem:Vmourre1} and \ref{lem:Vmourre2}, we have 
\begin{equation}
\label{mourr.30}
\la g_{\cI}(\fV_0) \, \im [ \fV_0, A] \, g_{\cI}( \fV_0)  u, u \ra \geq 
\rho  \la \Opw{g_\cI(\la v_0\ra)^2}  u, u \ra + \la (\fK - C H_0^{-2}) u, u \ra  
\end{equation}
with $\fK$ compact. 
Now by symbolic and functional calculus
\begin{align*}
\Opw{g_\cI(\la v_0\ra)^2} 
&= \left( \Opw{g_\cI(\la v_0 \ra)}\right)^2 + \cS^{-1} 
\stackrel{\eqref{g.dec}}{=} g_\cI(\fV_0)^2 + \fK_1
\end{align*}
with $\fK_1$ a compact operator.
 Inserting in \eqref{mourr.30} gives the claimed Mourre estimate \eqref{mourre.ver}.
\end{proof}

\section{Dynamics of the effective equation}	\label{sec:effective}
	In this section we  consider the effective equation obtained removing $R_N(t)$ from \eqref{res.eq}, namely 
\begin{equation}
\label{eq.H}
\im \pa_t \vf = H_N \vf, \qquad H_N := \Opw{\la v_0\ra} +   T_N,
\end{equation}
with $T_N \in \cS^{-\delta}$, $\delta >0$, see Proposition \ref{prop:rpdnf}.
Recall that $H_N$ is selfadjoint and time independent.
We shall construct a solution of  \eqref{eq.H} with {\em decaying  negative Sobolev norms}, and thus, exploiting the $L^2$ conservation, also with   growing Sobolev norms.

\begin{proposition}[Decay of negative Sobolev norms]
\label{prop:decay}
Consider
the  operator $H_N $  in \eqref{eq.H}.
For any  $ k  \in \N$, 
there exist a nontrivial solution $\vf(t) \in  \cH^k$ of \eqref{eq.H}   and   $\forall r \in [0, k]$ a constant $C_r >0$ such that 
\begin{equation}
\label{decay}
 \norm{ \vf(t)}_{{-r}}  \leq C_r \la t \ra^{-r} \,  \norm{\vf(0)}_r  \ , \qquad \forall t \in \R \ . 
\end{equation}
\end{proposition}

\begin{remark}
\label{rem:g}
 As $H_N$ is selfadjoint,  the conservation of the $L^2$-norm and 
Cauchy-Schwarz inequality give 
$$
\norm{\vf(0)}_0^2 = \norm{\vf(t)}_{0}^2 \leq \norm{\vf(t)}_{r} \ \norm{\vf(t)}_{{-r}}  \ , \qquad \forall t \in \R \ ,
$$
so  \eqref{decay} implies the growth of positive Sobolev norms:  
\begin{equation}
\label{low.vf}
\norm{\vf(t)}_r \geq  \frac{1}{C_r} \frac{\norm{\vf(0)}_0^2}{\norm{\vf(0)}_r} \, \la t \ra^{r}  \ , \quad  \forall t \in \R \ . 
\end{equation}
\end{remark}

Proposition \ref{prop:decay} will follow from the following abstract Sigal-Soffer local energy decay estimate:

\begin{theorem}[Local energy decay estimate]
\label{thm:SS}
Let $(\cH, \norm{\cdot}_\cH)$ be a Hilbert space. 
Let $\fH \in \cL(\cH)$ and $\fA$ be both selfadjoint 
and with 
$D(\fA)\cap \cH$  dense in $\cH$. 
Fix $k \in \N$ and 
assume that 
\begin{itemize}
\item[(M1)] \label{M1}  	the operators ${\rm ad}^n_\fA(\fH)$,  $ n = 1, \ldots,  4k +2$, 
can all be extended to bounded operators on $\cH$.
\item[(M2)] {\em Strict Mourre estimate}: there exist an open interval $\cI\subset \R$ with compact closure and a function
$g_\cI \in C^\infty_c(\R, \R_{\geq 0})$ with  $g_\cI \equiv 1$ on $\cI$ such that 
\begin{equation}
g_\cI(\fH) \, \im [\fH, \fA] \, g_\cI(\fH)  \geq \theta g_\cI(\fH)^2  
\end{equation}
for some $\theta >0$.
\end{itemize}
Then for any interval $\cJ \subset \cI$, any function $g_\cJ \in C^\infty_c(\R, \R_{\geq 0})$ with ${\rm supp }\, g_\cJ \subset \cI$, $g_\cJ = 1$ on $\cJ$,  there exists $C >0$ such that 
\begin{equation}
\label{SS}
\norm{ \la \fA \ra^{-k} \, e^{- \im \fH t} \, g_\cJ(\fH) \, \psi}_\cH \leq C \la t \ra^{-k} \norm{ \la \fA \ra^{k}\, g_\cJ(\fH) \psi}_\cH  , \quad \forall t \in \R  \ .
\end{equation}
\end{theorem}
The  theorem goes back to the works of \cite{SS,Ski, JenMouPer}, see also \cite{HunSigSof,Bach99,DG,GerSig,GNRS} for extensions and generalizations.
A  proof of this exact statement can be found in \cite{Mas21}, Appendix C.

\begin{proof}[Proof of  Proposition \ref{prop:decay}]
We apply Theorem \ref{thm:SS} with $\cH = L^2$, $\fH = H_N$ in \eqref{eq.H} and $\fA = A$ in \eqref{def:A}.
 Assumption (M1) is verified since $A \in \cS^1$ and $H_N \in \cS^0$. 
 To check (M2) 
we   work perturbatively from the Mourre estimate \eqref{mourre.ver}.
Again we shall denote by $\fK$ a general compact operator in $L^2$ which might change from line to line. 
To shorten notation, we shall also write $\fV_0:= \Opw{\la v_0\ra}$.

 First, as $H_N =\fV_0 + T_N$ with $T_N \in \cS^{-\delta}$, the operator
 $[T_N, A] \in \cS^{-\delta}$ is compact, so from \eqref{mourre.ver} we get 
  $$
 g_\cI(\fV_0) \,  \im [H_N, A] \, g_{\cI}(\fV_0) \geq \rho \, 
  g_\cI(\fV_0)^2 + \fK
 $$
 with $\fK$ compact.
Next, by Lemma \ref{lem:fc}, $g_\cI(H_N) - g_\cI(\fV_0)$ is compact and therefore we get that 
 \begin{align*}
  g_\cI(H_N) \,  \im [H_N, A] \, g_{\cI}(H_N) 
& =
    g_\cI(\fV_0) \,  \im [H_N, A] \, g_{\cI}(\fV_0) + \fK  \\
 &  \geq \rho \, 
  g_\cI(\fV_0)^2 + \fK  =  \rho \,  g_\cI(H_N)^2 + \fK
 \end{align*}
This proves that  $H_N$ fulfills a Mourre estimate over $\cI$.
 It is standard that one can shrink the  interval $\cI$ to a subinterval $\cI_1$ to obtain the strict Mourre estimate 
 \begin{equation}
\label{M3}
g_{\cI_1}(H_N) \, \im [H_N, A] \, g_{\cI_1}(H_N)  \geq \frac{\rho}{2} \,  g_{\cI_1}(H_N)^2   \ ,
\end{equation}
see e.g. the arguments in \cite{Mourre} (or also Step 1 of Lemma 3.13 of \cite{Mas21}).
This proves (M2). 
So we apply 
Theorem \ref{thm:SS} and obtain that for any interval $\cJ \subset \cI_1$, any   function $g_\cJ \in C^\infty_c(\R, \R_{\geq 0})$ with ${\rm supp }\, g_\cJ \subset \cI_1$, $g_\cJ = 1$ on $\cJ$, 
\begin{equation}
\label{SS.applied}
\norm{ \la A \ra^{-k} \, e^{- \im H_N t} \, g_\cJ(H_N) \, \vf}_0 \leq C \la t \ra^{-k} \norm{ \la A \ra^{k}\, g_\cJ(H_N) \vf}_0  , \quad \forall t \in \R .
\end{equation}
Then, since  $H_0^{-k} \la A \ra^{k}, \la A \ra^{k} H_0^{-k} \in \cS^0$, using \eqref{SS.applied} we deduce 
\begin{align*}
 \norm{ e^{- \im t H_N} g_\cJ(H_N) \vf}_{-k}
&\leq C_{k} \la t \ra^{-k} \norm{g_\cJ(H_N) \vf}_{k} \ , 
\end{align*}
and interpolating with 
 $\norm{  e^{- \im t H_N} \vf}_0 = \norm{\vf}_0$ yields
$$
 \norm{  e^{- \im t H_N} g_\cJ(H_N) \vf}_{-r}  \leq  C_r  \la t \ra^{-r} \norm{g_\cJ(H_N) \vf}_r \ , \quad \forall t \in \R \ , \ \ \ \forall  \vf \in \cH^r \  , \quad \forall r \in [0,k] \ .
$$
The last step is to prove that the estimate is not  trivial, namely  that one can choose   $\vf(0):= g_\cJ(H_N)\vf \in \cH^r\setminus\{0\}$. 
Regarding the regularity, note that  $g_\cJ(H_N)\vf \in \cH^r$ provided $\vf \in \cH^r$.
 Hence it suffices to show that  $g_\cJ(H_N) \cH^r \neq \{0\}$.
By the density of $\cH^r$ in $\cH\equiv L^2$, this follows provided $g_\cJ(H_N) \cH \neq \{0\}$.
So assume by contradiction that $g_\cJ(H_N) \cH = \{0\}$. 
By Weyl's theorem, being $T_N$ compact,
$$
\sigma(H_N) \supseteq \sigma_{ess}(H_N) = 
\sigma_{ess}(\fV_0)\stackrel{ \eqref{I0}}{\supseteq} \cI_0\supseteq \cJ \ .
$$
Hence  the spectral projector $E_\cJ(H_N)$ of $H_N$ over $\cJ$ fulfills $E_\cJ(H_N)\cH \neq \{0\}$. 
Since,  by functional calculus, $E_\cJ(H_N) = E_\cJ(H_N) g_\cJ(H_N)$
(which follows from   $\mathds{1}_{\cJ}(\lambda) = \mathds{1}_\cJ(\lambda) g_\cJ(\lambda)$ $\, \forall \lambda \in \R$, $\mathds{1}_{\cJ}$ being the indicator function over the interval $\cJ$), we get
$$
\{0 \} =  E_\cJ(H_N) g_\cJ(H_N)\cH =  E_\cJ(H_N) \cH \neq \{0\}
$$
obtaining a contradiction. Hence $g_\cJ(H_N) \cH \neq  \{0\}$.
\end{proof}

\section{Proof of the main  theorem}
In this section we prove Theorem  \ref{thm:main}.
\paragraph{Proof of $(i)$.}
It follows exactly as in \cite{Mas21}, so we just sketch the arguments for completeness. 
Let  $r >0$ be given in Theorem \ref{thm:main}. 
Fix $ N, k \in \N$ with $N \geq {2r +2}$ and $k \geq \delta +N-r$.
Take a solution
  $\vf(t)$ of equation \eqref{eq.H} fulfilling 
  the decay \eqref{decay} up to regularity $k$.
Defining   $U_N(t,s)$  the linear propagator of $H_N + R_N(t)$, one checks that
$$
\phi(t) := \vf(t) + \im  \int\limits_t^{+ \infty} U_N(t,s) \, R_N(s) \, \vf(s) \, \di s  =: \vf(t) + w(t)
$$
  solves equation  \eqref{res.eq}, provided 
  $w(t)$ is well defined. 
Using that,   by Theorem 1.5  of \cite{MaRo}, 
$U_N(t,s) \in \cL(\cH^r)$ with 
 $\norm{U_N(t,s)}_{\cL(\cH^r)} \leq C_r \, \la t-s \ra^{r} , \quad \forall t,s \in \R$ and  
 $R_N \in C^0(\T, \cS^{-N-\delta}) $, we get 
 \begin{align*}
 \norm{w(t)}_r 
 &
 \leq C  \int\limits_t^{+ \infty}   \la t-s \ra^r \norm{R_N(s) \, \vf(s)}_r \, \di s 
 \leq
 C  \int\limits_t^{+ \infty}   \la t-s \ra^r  \, \norm{\vf(s)}_{-(\delta+ N-r)} \, \di s\\
 & {\leq} C  \, \norm{\vf_0}_{\delta+ N - r}  \int\limits_t^{+ \infty}   \la t-s \ra^r  \, \frac{1}{\la s \ra^{\delta+ N-r}} \,  \, \di s
 \leq 
  C  \, \norm{\vf_0}_{k}   \la t \ra^{-1} \ .
 \end{align*}
Using also \eqref{low.vf} we deduce that  $\norm{\phi(t) }_r \geq \norm{\vf(t)}_r - \norm{w(t)}_r \geq C \la t \ra^r$ for $t$ sufficiently large.

Finally 
$\psi(t) = \cU_N(t)^{-1} \phi(t)$ solves \eqref{har.osc} and has  polynomially growing Sobolev norms as \eqref{thm:est},  proving item $(i)$.

\paragraph{Proof of $(ii)$.} We show that the set $\sV$ is  open and dense in $C^0_r(\T, S^0_{\cl})$.

\noindent\underline{Open:}  We show that for any  $v \in \sV$, there is 
 $\e, M >0$ such that any 
$ w \in C^0_r(\T, S^0_{\cl})$ fulfilling $\wp^{0,0}_M(v-w) <  \e$ belongs to $\sV$. In particular this last condition is achieved provided 
  $\td^{0,0}(v,w)$ is small enough.

First of all decompose   $v = v_0 + v_{-\mu}$ with $v_0$ the principal symbol of $v$ and $v_{-\mu} \in C^0_r(\T, S^{-\mu})$, $\mu >0$. 
As $\la v_0\ra$ is positively homogeneous of degree  0 (see \eqref{v0-hom}), we put
$$
\varrho := \max_{x^2 + \xi^2 \geq 1} \abs{ \{ \la v_0 \ra, h_0\}} = \max_{\vartheta \in \T} \abs{\pa_\vartheta \widetilde{\la v_0 \ra}(\vartheta,1)} >0 \ ;
$$
note that $\varrho$ is strictly positive since, by assumption,  $\{ \la v_0 \ra, h_0\} $ is not identically 0.
  Denote by $\vartheta_\varrho$ a point in $\T$ where the maximum is attained.

Decompose also $w = w_0 + w_{-\mu'}$ with $w_0$ the principal symbol
and $w_{-\mu'} \in C^0_r(\T, S^{-\mu'})$, $\mu' >0$.
To prove that $w \in \sV$, it suffices to show that $\{ \la w_0\ra, h_0\}$ is not identically zero in $x^2 + \xi^2 \geq  1$. 
We write
\begin{align}
\notag
\{ \la w_0\ra, h_0\}
&= \{ \la v_0 \ra, h_0\} + \{ \la w_0 - v_0 \ra, h_0\} \\
\label{openV2}
& =  \{ \la v_0 \ra, h_0\}
+ \{ \la w - v \ra, h_0\}  - \{ \la w_{-\mu'} - v_{-\mu}\ra, h_0 \}
\end{align}
We evaluate \eqref{openV2} at the point
\begin{equation}
\label{openV1}
\und{x}:= \sqrt{2\tR} \sin(\vartheta_\delta) , \quad \und{\xi}:= \sqrt{2\tR} \cos(\vartheta_\delta)
\end{equation}
where $\tR>0$ will be chosen later on sufficiently large.\\
By the very definition of $(\und x, \und \xi)$ we have
\begin{equation}
\label{openV3}
 \{ \la v_0 \ra, h_0\}(\und x, \und \xi) = \varrho \ .
\end{equation}
Next we consider the second term of \eqref{openV2}; by  symbolic calculus
\begin{equation}
\label{openV4}
\sup_{x,\xi \in \R}\abs{ \{ \la w - v \ra, h_0\}( x, \xi)  } \leq C \wp^{0,0}_M(w-v) \leq C \e   \ . 
\end{equation}
Finally let us consider $\{ \la w_{-\mu'} - v_{-\mu}\ra, h_0 \}$.
By  Remark \ref{rem:res.av} this is a symbol in $S^{-\bar \mu}$, $\bar \mu:= \min(\mu, \mu')$.
 By definition 
 there exists a constant $C_1 = C_1(w_{-\mu'}, v_{-\mu}) >0$ such that
$$
|\{ \la w_{-\mu'} - v_{-\mu}\ra, h_0 \}(x,\xi)|  \leq C_1 (1 + x^2 + \xi^2)^{-\bar \mu}  , \quad \forall (x, \xi) \in \R^2 \ . 
$$  
In particular, at the point $(\und x, \und \xi)$ in \eqref{openV1}, we get that
\begin{equation}
\label{openV5}
|\{ \la w_{-\mu'} - v_{-\mu}\ra, h_0 \}(\und x, \und \xi)|  \leq \frac{C_1}{(1+2\tR)^{\bar \mu}} \ . 
\end{equation}
Thus, evaluating \eqref{openV2} at the point $(\und x, \und \xi)$ and using \eqref{openV3}, \eqref{openV4} and \eqref{openV5} we get
$$
\{ \la w_0\ra, h_0\}(\und x, \und \xi) \geq \varrho - C \e - \frac{C_1}{(1+2\tR)^{\bar \mu}} \geq \frac{\varrho}{2} > 0
$$
provided one chooses $\e < \dfrac{\varrho}{4 C}$ and $\tR> \left(\dfrac{ C_1}{2\varrho}\right)^{\frac{1}{\bar \mu}}$, concluding the verification that $w \in \sV$.

\noindent\underline{Dense:} 
Take $v \in C^0_r(\T, S^0_{\cl})$ and assume that $v \not\in\sV$. Take $\e >0$ arbitrary. 
We shall construct $w \in  \sV$ with $\td^{0,0}(v,w) < \e$. 
Pick a function $\eta \in C^\infty_c(\R^2, \R_{\geq 0})$, radial, and with $\eta(x,\xi) = 1$ for $x^2  + \xi^2 \geq \frac12$ and 
$\eta(x,\xi) = 0$ for $x^2  + \xi^2 \leq \frac14$.
We put
\begin{equation}
\label{}
w:= v + \e_0 w_0 , \quad w_0(t,x,\xi):= \cos(2t)\, \eta(x,\xi)\,  \frac{x\xi}{x^2+\xi^2}
\end{equation}
with $\e_0>0$ small enough. Note that $w_0$ is positively homogeneous of degree 0, so $w \in C^0_r(\T, S^0_{\cl})$.
Let us show that $w \in \sV$.  Writing 
$w = v_0 + \e_0 w_0 + v_{-\mu}$, we have to check that 
$\{\la v_0\ra + \e_0 \la w_0 \ra , h_0\} \not\equiv 0$ in $x^2+\xi^2 \geq 1$. 
Since $v \not\in\sV$, we have that $\{\la v_0\ra  , h_0\} \equiv 0$ in $x^2+\xi^2 \geq 1$ and we need only to check that $ \{\la w_0 \ra , h_0\} \not\equiv 0$ on the same set.
Computing explicitly  $\la w_0\ra$ (using also that $\eta$ is radial and thus  constant along the flow $\phi^t$), we obtain 
\begin{equation}
\label{}
\la w_0 \ra = 
\frac{\eta(x,\xi)}{x^2+\xi^2} \cdot \frac{1}{2\pi} \int_0^{2\pi} \cos(2t) \, \big(x \cos t + \xi \sin t \big)\big(-x \sin t + \xi \cos t\big) \di t  = \frac{1}{2} \eta(x,\xi)\,  \frac{x\xi}{x^2+\xi^2} \ . 
\end{equation}
Then, using that $\{\dfrac{\eta(x,\xi)}{x^2+\xi^2}, h_0 \} = 0$, we get 
$$
\{ \la w_0 \ra, h_0 \} =\frac12 \{x\xi, h_0 \} \frac{\eta(x,\xi)}{x^2+\xi^2} = \frac12 (x^2 - \xi^2)  \frac{\eta(x,\xi)}{x^2+\xi^2} 
$$
proving that $w \in \sV$.
Finally we show that, provided $\e_0 >0$ is sufficiently small, 
$\td^{0,0}(v,w) < \e$. 
So take $N>0$ so large that $\sum_{j \geq N+1} 2^{-j} < \e/2$
and $\e_0$ so small that
$$\sum_{j = 0}^N \wp_j^{0,0}(v-w)  = \e_0 \sum_{j = 0 }^N \wp_j^{0,0}(w_0) < \frac{\e}{2} \ .$$
Then
$\td^{0,0}(v,w)   \leq \sum_{j = 0}^N \wp_j^{0,0}(v-w) + 
 \sum_{j \geq N+1} \frac{1}{2^{j}} <  \e $.
\qed

\vspace{1em}

We conclude this part with the following lemma, which somehow generalize the construction of the symbol $w$ in the previous proof.

\begin{lemma}\label{lem:ex1}
Let  $ \tv \in S^0_\cl $ be  real valued and so that $\{h_0, \tv \} \not\equiv 0$ in $x^2 + \xi^2 \geq 1$.
Then there exists $n \in \N$ such that 
\begin{equation}
\label{V.har0}
v(t,x, \xi):=  \cos( n t) \, \tv(x, \xi)  \in \sV \ . 
\end{equation}
In particular $\Opw{v(t, \cdot)}$  is a transporter for 
\eqref{har.osc}.
\end{lemma}
\begin{proof}
We need to check that $\{h_0, \la v \ra \} \not\equiv 0$ in $x^2 + \xi^2 \geq 1$, where
$\la v \ra(x,\xi) = \frac{1}{2\pi}\int_0^{2\pi} \cos(nt) \, \tv(\phi^t(x,\xi)) \, \di t$.
The function $t \mapsto \tv(\phi^t(x,\xi))$ is $2\pi$-periodic in time; expanding it in Fourier series one gets
\begin{equation}
\label{tv.fou}
\tv(\phi^t(x,\xi)) = \tv^+_0(x, \xi) + 2  \sum_{n \geq 1} \big(\cos(nt)\, \tv^+_n(x,\xi) +  \sin(nt)\, \tv^-_n(x,\xi)\big)    
\end{equation}
where
$$
\tv^+_n(x, \xi) := \frac{1}{2\pi} \int\limits_0^{2\pi} \cos(nt) \,  \, \tv(\phi^t(x,\xi)) \, \di t , 
\quad
\tv^-_n(x, \xi) := \frac{1}{2\pi} \int\limits_0^{2\pi} \sin(nt) \, \tv(\phi^t(x,\xi)) \, \di t , \quad n \geq 0 . 
$$
They are all symbols positively  homogeneous of degree 0.
The claim is equivalent to verify that 
\begin{equation}
\label{tvn+1}
\exists \ n \geq 1 \colon \ \ \{h_0, \tv^+_n \} \not\equiv 0  \ \ \mbox{ in }  x^2 + \xi^2 \geq 1  \ . 
\end{equation}
First note that, by  integration by parts and the periodicity of the flow $\phi^t$, 
\begin{align}\notag
\{ h_0,  \tv^+_n \} & = 
\frac{1}{2\pi}\int\limits_0^{2\pi} \cos(nt) \, \{ h_0,  \tv\circ \phi^t \} \, \di t = 
\frac{1}{2\pi}\int\limits_0^{2\pi} \cos(nt) \, \left( \frac{\di }{\di t} \tv\circ \phi^t \right) \, \di t \\
\label{tvn+11}
& = \frac{n}{2\pi}\int\limits_0^{2\pi} \sin(nt) \,  ( \tv\circ \phi^t ) \, \di t = n\,  \tv^-_n  \, , 
\end{align}
so it is sufficient to check that $\exists n \geq 1$ so that $\tv^-_n \not\equiv 0$ in $x^2 + \xi^2 \geq 1$.
Actually, since 
\begin{align}\notag
\tv_n^-(x,\xi) & = 
\frac{1}{2\pi} \int_0^{2\pi} \sin(nt) \, \tv\big(\phi^t(x,\xi)\big) \, \di t  
= 
\frac{1}{2\pi} \int_0^{2\pi} \cos\big(n(t-\frac{\pi}{2n})\big) \, \tv\big(\phi^t(x,\xi)\big) \, \di t \\
\label{tvn+12}
&  = 
\frac{1}{2\pi} \int_0^{2\pi} \cos(nt) \, \tv\big(\phi^{t+\frac{\pi}{2n}}(x,\xi)\big) \, \di t  = \tv^+_n\big(\phi^{\frac{\pi}{2n}}(x,\xi) \big) , 
\end{align}
it is enough to check that $\exists n \geq 1$ so that  one among $\tv^\pm_n$ is not identically zero in $x^2 + \xi^2 \geq 1$; this is what we show next. 

Assume by contradiction that $\tv^\pm_n \equiv 0$ in $x^2 + \xi^2 \geq 1$ for any  $n \geq 1$; then  from \eqref{tv.fou} we get  $\tv\circ\phi^t = \tv^+_0$ for any  $ t \in \R$ and any $x^2 + \xi^2 \geq 1$. Then
$$
0 = \frac{\di}{\di t} \tv\circ \phi^t = \{ h_0, \tv  \} \circ \phi^t  \ , \quad \forall t  \in \R , \quad \forall x^2 + \xi^2 \geq 1 ; 
$$
so, at $t = 0$, one gets $ \{ h_0, \tv  \} \equiv 0$ in $x^2 + \xi^2 \geq 1$, contradicting the assumption. 

Then at least one couple of the   $\tv_n^\pm$ is not identically zero, 
and  \eqref{tvn+1} follows.
\end{proof}

\appendix

\section{Technical results}

\subsection{The strong G\r{a}rding inequality in $\cS^0$}\label{app:Garding}
Our proof involves the  Anti-Wick quantization of a symbol, which we now introduce.
First let us define  coherent states:
for $z = (q,p) \in \R^2$ let
\begin{equation}\label{def:vf}
\begin{aligned}
\Phi_0(x) := \frac{1}{\pi^{1/4}} \, e^{-\frac{x^2}{2}} , \qquad 
\Phi_z := \cT_z \Phi_0 , 
\qquad 
[\cT_z u](x):= e^{- \frac{\im}{2} pq} \, e^{ \im x p} \, u(x-q) \ . 
\end{aligned}
\end{equation}
Note that the operator $\cT_z$ is unitary in $L^2(\R)$.
Now, given a symbol $a \in S^0_\hos$, we define its Anti-Wick quantization by
\begin{equation}
\label{aw}
\left(\Opaw{a}u \right)(x) := \int_{\R^{2}} a(q,p) \, \la u, \Phi_{q,p} \ra\, \Phi_{q,p} \, \di q \, \di p \ .
\end{equation}
We collect few properties of the Anti-Wick quantization:
\begin{lemma}\label{lem:aw}
 Let $a \in S^0_\hos$. Then
\begin{itemize}
\item[(i)] If $a \geq 0$, then $\la \Opaw{a} u, u \ra \geq 0$ for any $u \in L^2$.
\item[(ii)] One has $\Opaw{a} = \Opw{a*\Phi_0}$, with $\Phi_0$ in \eqref{def:vf}. Moreover $a - a*\Phi_0 \in S^{m-2}_\hos$.
\end{itemize}
\end{lemma}
\begin{proof}
Item $(i)$ follows directly from the  definition \eqref{aw}. Item $(ii)$ is   Theorem 24.1 of \cite{Shubin}.
\end{proof}

\begin{proof}[Proof of Theorem \ref{thm:garding}]
Let $\chi\in C^\infty_c(\R^2, \R_{\geq 0})$, radial  cut-off function with
$$
\chi(x,\xi) = 1  \ \ \forall x^2 + \xi^2 \leq R , \qquad
\chi(x,\xi) = 0 \ \ \forall x^2 + \xi^2 \geq R+1 \ .
$$
Then the function $a_\chi:= a(1-\chi) \in S^0_\hos$ fulfills, using the assumption \eqref{garding.a.ass},
\begin{equation}
\label{achi}
a_\chi(x,\xi) \geq 0 \quad \forall x,\xi \in \R \ . 
\end{equation}
Next we write
\begin{align*}
\Opw{a_\chi}  
& = \Opw{a_\chi * \Phi_0} + \Opw{a_\chi - a_\chi*\Phi_0}\\
& = \Opaw{a_\chi} +   \Opw{a_\chi - a_\chi*\Phi_0} 
\end{align*}
where to pass from the first to the second line we used Lemma \ref{lem:aw} $(ii)$. By Lemma \ref{lem:aw} and \eqref{achi} one has   $ \Opaw{a_\chi} \geq 0$, hence 
$$
\la \Opw{a_\chi} u , u \ra \geq \la  \Opw{a_\chi - a_\chi*\Phi_0}  u, u \ra \ .
$$
Now use that $a_\chi = a - a\chi$ to deduce
$$
\la \Opw{a} u , u \ra \geq \la  \Opw{a_\chi - a_\chi*\Phi_0}  u, u \ra  + \la \Opw{a\chi} u , u \ra . 
$$
By Lemma \ref{lem:aw} $(ii)$ the symbol $a_\chi - a_\chi*\Phi_0 \in S^{-2}_\hos$ and moreover $a\chi \in S^{-\infty}_\hos$. 
Then Theorem \ref{thm:sym.cal} $(i)$  implies the claimed bound 
\eqref{sGarding}.
\end{proof}

\subsection{Proof of Theorem \ref{thm:spectrum_ho} }\label{app:spectrum}

Denote by $\sA$ the set on the right of \eqref{spectrum}.
 We first show that $\sigma_{ess}(\Opw{v}) \subseteq \sA$.
 Assume by contradiction that
 $\lambda \in \sigma_{ess}(\Opw{v})$ does not belong to 
 $\sA$. 
 Then there exist $c, R >0$ such that $|v(x,\xi)-  \lambda|\geq c$
 for any $x^2+\xi^2 \geq R$.
Let $\chi \in C^\infty_c(\R^2)$ with $\chi \equiv 1$ in $x^2+\xi^2 \leq R$ and $\chi \equiv 0$ in $x^2+\xi^2 \geq R+1$. 
Put 
  $b(x,\xi) = \dfrac{1-\chi(x,\xi)}{v(x,\xi) -\lambda} \in S^0_\hos$.
Then by symbolic calculus there are operators $K_1, K_2 \in \cS^{-1}$ (and therefore compact) so that 
$$
\left( \Opw{v} - \lambda \right) \, \Opw{b} = \textup{Id} + K_1 , \qquad
\Opw{b}\, \left( \Opw{v} - \lambda \right)  = \textup{Id} + K_2  \ . 
$$
Thus  $ \Opw{v} - \lambda $ is a Fredholm operator, and its 
  spectrum in a neighbourhood of zero must be discrete. 
  In particular $\lambda \not\in \sigma_{ess}(\Opw{v})$, contradicting the assumption. 
This shows that $\sigma_{ess}(\Opw{v}) \subseteq \sA$.

We show now the inverse inclusion  $\sA \subseteq \sigma_{ess}(\Opw{v})$. Given $\lambda \in \sA$, we shall 
 exhibit a Weyl's sequence for $\lambda$, i.e. a sequence  of functions $\{\varphi_j \}_{j \geq 1}\in L^2$ with  $\norm{\varphi_j} = 1$, $\varphi_j \rightharpoonup 0$ and $\norm{(\Opw{v}- \lambda) \varphi_j} \to 0$; then 
Weyl's theorem guarantees that
$\lambda \in \sigma_{ess}(\Opw{v})$.
We construct such a  Weyl's sequence using  the coherent states in \eqref{def:vf}.
If  $\lambda \in \sA$  we have a  sequence $z_j:=(q_j, p_j) \to \infty$ with $v(q_j, p_j) \to \lambda$. 
Put (see \eqref{def:vf})
$$
\Phi_j := \Phi_{z_j} = \cT_{z_j} \Phi_0  \ .
$$
We claim that, up to subsequences, $\{\Phi_j\}_{j \geq 1}$ is a Weyl sequence for $\lambda$. 
It is clear that $\norm{\Phi_j} = 1$  $\forall j$. 
It is not difficult to show that  $\Phi_j \rightharpoonup 0$; for completeness we prove this in Lemma \ref{lem:vf.to0} below. 
We show  now that $\norm{(\Opw{v}- \lambda) \Phi_j} \to 0$.
First we have, using the reality of $\lambda$, $v$ and symbolic calculus, 
\begin{align*}
\Opw{v-\lambda}^* \Opw{v- \lambda} = \Opw{(v-\lambda)^2 + r_{-1}} , \quad r_{-1} \in S^{-1}  \ .
\end{align*}
 So we write
\begin{align*}
\norm{(\Opw{v}- \lambda) \Phi_j}^2
&  = 
{\la  \Opw{(v-\lambda)^2} \Phi_j, \Phi_j \ra} 
+
\la \Opw{r_{-1}} \Phi_j, \Phi_j \ra  \ .
\end{align*}
Since $\Opw{r_{-1}} \in \cS^{-1}$ is compact and $\Phi_j \rightharpoonup 0$, it follows that $\la \Opw{r_{-1}} \Phi_j, \Phi_j \ra  \to 0$ and therefore it  suffices to show   that
$\la  \Opw{(v-\lambda)^2} \Phi_j, \Phi_j \ra \stackrel{j \to \infty}{\to} 0$.
To estimate  this last term  we shall use the following identities:
\begin{align}
\label{weyl1}
& \cT_z^{-1} \, \Opw{a} \, \cT_z = \Opw{a(\cdot + z)}  \ , \\
\label{weyl2}
& \la \Opw{a} \Phi_0, \Phi_0 \ra = \frac{1}{\pi} \int_{\R^2} a(x,\xi) \, e^{-(x^2 +\xi^2) } \, \di x \, \di \xi .
\end{align}
They are both contained in the book \cite{CoRo}: the 
 first one is   formula (2.16), whereas the  second one follows  combining Proposition 14 and 16 (with $\hbar =1$) in chap. 2. \\
We deduce, using $\cT_{z_j}^* = \cT_{z_j}^{-1}$, 
\begin{align*}
\la \Opw{(v-\lambda)^2} \Phi_j, \Phi_j \ra 
& = \la \cT_{z_j}^{-1} \, \Opw{(v-\lambda)^2} \,\cT_{z_j} \Phi_0, \Phi_0 \ra \\
&  \stackrel{\eqref{weyl1}}{=}  
\la  \Opw{\left( v(\cdot + z_j) - \lambda \right)^2}  \Phi_0, \Phi_0 \ra  \\
& \stackrel{\eqref{weyl2}}{=}\pi^{-1} \int_{\R^2} \big( v(x + q_j,\xi+p_j) - \lambda \big)^2  \, e^{-(x^2 +\xi^2) } \, \di x \, \di \xi 
\end{align*}
The last integral converges to $0$ as $j \to \infty$ 
by Lebesgue's dominated convergence theorem, since $v \in S^0_\hos$ is bounded and  fulfills $v(x + q_j,\xi+p_j) \stackrel{j \to \infty}{\to} \lambda$ pointwise when $(q_j,p_j) \to \infty$.
\qed

\begin{lemma}\label{lem:vf.to0}
Let $z_j:=(q_j,p_j) \to \infty$. 
Then, up to a subsequence,   $\Phi_j:= \cT_{z_j} \Phi_0 \rightharpoonup 0$.
\end{lemma}
\begin{proof}
We distinguish  two cases:
$(i)$ the sequence $\{q_j\}_j$ is bounded and 
$(ii)$ up to a subsequence  $|q_j| \to \infty$.\\
In case $(i)$, up to a subsequence  we can assume $q_j \to  q_0 \in \R$ and $p_j \to \infty$. 
Take an arbitrary $f \in L^2(\R)$. We write 
  \begin{align}
  \label{vf.w0.1}
\int_\R f(x) \, \bar{\Phi_j(x)} \, \di x  = 
\underbrace{\int_\R f(x) \, \bar{\Phi_{(q_0, p_j)}(x)} \di x }_{=: \tI_1}
+
\underbrace{\int_\R f(x) \, \bar{\left(\Phi_{( q_j, p_j)}(x) -
\Phi_{( q_0, p_j)}(x) \right)} \di x }_{=:\tI_2}  \ .
 \end{align}
Using  Riemann-Lebesgue lemma
 \begin{equation}
 \label{vf.w0.2}
\tI_1   =  \frac{e^{\frac{\im}{2} p_j q_0}}{ \pi^{\frac14}} \int_\R f(x)  \, e^{-\frac{(x-q_0)^2}{2}} \,  e^{-\im x p_j} \di x \stackrel{p_j \to \infty}{\to} 0 \ .
 \end{equation}
To estimate the  second integral in \eqref{vf.w0.1} we write
\begin{align*}
\tI_2 & = \pi^{-\frac{1}{4}} \int_\R
e^{-\frac{x^2}{2}} \, e^{- \im x p_j} \,
\left(f(x+q_j) e^{-\frac{\im}{2}p_j q_j} - f(x+q_0) e^{-\frac{\im}{2}p_j q_0} \right) \, \di x\\
& = \underbrace{\frac{e^{-\frac{\im}{2} p_j q_j}}{\pi^{\frac14}}
\int_\R e^{-\frac{x^2}{2}} \, e^{-\im x p_j} \,  \big(f(x+q_j) - f(x+q_0)\big) \, \di x }_{\tI_{21}} +
\underbrace{\frac{e^{-\frac{\im}{2}p_j q_j} - e^{-\frac{\im}{2}p_j q_0}}{\pi^{\frac14}}
\int_\R e^{-\frac{x^2}{2}} \, e^{-\im x p_j} \, f(x+q_0) \, \di x }_{\tI_{22}}
\end{align*}
Next we have that
$$
\abs{\tI_{21}} \leq 
C \norm{ f(\cdot+q_j) - f(\cdot+q_0)}_{L^2(\R)} \,  
\stackrel{q_j \to q_0}{\to} 0
$$
by the continuity of the translations in $L^2(\R)$, whereas
$$
\abs{\tI_{22}} \leq C \abs{\int_\R e^{-\frac{x^2}{2}} \, e^{-\im x p_j} \, f(x+q_0) \, \di x } \stackrel{p_j \to \infty}{\to} 0
$$
by Riemann-Lebesgue lemma. 
In conclusion we have proved that 
$\int_\R f(x) \, \bar{\Phi_j(x)} \, \di x  \to 0$ as $j \to \infty$. This concludes the proof of case $(i)$.

In case $(ii)$ we can assume that, up to a subsequence, $q_j$ has always the same sign.  Let  $f \in L^2(\R)$. Then one easily shows that 
 \begin{align*}
\abs{ \int_\R f(x) \, \bar{\Phi_j(x)} \di x }
 \leq \pi^{-1/4} \int_\R \abs{f(x+q_j)} \, e^{-\frac{x^2}{2}} \di x \stackrel{|q_j| \to \infty}{\to}  0 
 \end{align*}
 concluding the proof of case $(ii)$.
\end{proof}

\footnotesize

\bibliographystyle{plain} 

\end{document}